\documentclass[smallextended]{svjour3oma}       
\smartqed  
\usepackage{amssymb,amsmath}
\usepackage[T1]{fontenc}
\usepackage[utf8]{inputenc}
\usepackage[dvips]{graphicx}
\usepackage{psfrag}
\usepackage{enumerate}
\usepackage{stmaryrd}

\newcommand{\diam}[1]{\diamond_{#1}}
\newcommand{\s}[1]{\bar{#1}}

\newcommand{\anti}[1]{\overset{{}_{\shortleftarrow}}{#1}}
\newcommand{\alp}[1]{\tau(#1)}

\newcommand{\rwrule}{\longrightarrow}
\newcommand{\rewrite}{\rightarrow}
\newcommand{\conjugate}{\leadsto}
\newcommand{\derive}{\Rightarrow}

\newcommand{\Z}{\mathbb{Z}}
\newcommand{\N}{\mathbb{N}}
\journalname{myjournal}
\begin{document}

\title{Sub Rosa, a system of quasiperiodic rhombic substitution tilings with $n$-fold rotational symmetry
}

\titlerunning{Sub Rosa}        

\author{\mbox{Jarkko Kari \and Markus Rissanen\thanks{This research was supported during 2012 by a grant from the Kone Foundation}}}


\institute{Jarkko Kari \at
              Department of Mathematics and Statistics, FI-20014 University of Turku, Finland \\
              \email{jkari@utu.fi}           
           \and
           Markus Rissanen \at
             T\"o\"ol\"onkatu 11 A 130,
FI-00100 Helsinki, Finland\\
Tel: +358-40-7602020\\
\email{markus.rissanen@gmail.com}
}

\maketitle

\begin{abstract}
In this paper we prove the existence of quasiperiodic rhombic substitution tilings with $2n$-fold rotational symmetry, for any $n$.
The tilings are edge-to-edge and use $\lfloor{\frac{n}{2}\rfloor}$ rhombic prototiles with unit length sides. We explicitly describe the
substitution rule for the edges of the rhombuses, and prove the existence of the corresponding tile substitutions by proving that the
interior can be tiled consistently with the given edge substitutions.

\end{abstract}

\section{Introduction}
\label{intro}

A tiling is a covering of the infinite plane using copies of a finite number of different prototiles, without leaving gaps or letting the tiles overlap.
In this paper we consider edge-to-edge substitution tilings of the plane with rhombuses whose edges are of unit length.
The edge-to-edge condition means that two tiles in the tiling are either disjoint, have one common vertex or have one common edge.
A substitution rule tells how to replace enlarged tiles by patches of tiles. By iterating the process of inflating the tiles and substituting
the corresponding patches, one obtains a sequence whose limit is a tiling of the infinite plane. Such substitutions are
a popular way to produce non-periodic tilings. The Penrose tilings~\cite{Gardner,Penrose}
are the best known examples of this effect.

We say that a tiling has $n$-fold \emph{rotational symmetry\/} (with center $P$) if the tiling is invariant under the rotation around $P$ by the angle $\frac{2\pi}{n}$. Note that this requires a perfect global symmetry of the entire tiling.
There may exist an infinite number of centers of rotation for $n$-fold rotations only if $n=2,3,4$ or $6$. For all other values there can be at most one central point of
rotational symmetry. This fact is known as the \emph{crystallographic restriction}.

All our tilings are by rhombic prototiles with unit length sides.
As usual we shall define the substitution process in two steps. A tile is first enlarged by a constant called the \emph{scaling factor}. The enlarged tile is then replaced by a patch of prototiles that cover the same area as the enlarged tile does. This second step is called the \emph{substitution rule}. The boundaries of this new replaced tile do not have to go along the boundaries of the merely enlarged tile but their areas have to be equal. If there are some kind of ``dips'' inwards in the perimeter of the replaced tile also corresponding ``dips'' outwards have to exist and in the right places, so that no overlapping occurs along the edge of two neighbouring enlarged tiles. This condition
can be guaranteed by an associated \emph{edge substitution rule}.

Starting with any tile, the substitution process may be iterated to obtain higher order images.
The substitution is {\em primitive\/} if there exists $k$ such that the $k$'th order image of every tile contains a copy of every prototile.
All our substitutions will be primitive with $k=1$, so that each tile appears in the image of each tile.

A {\em substitution tiling\/} is any non-overlapping covering of the plane by the tiles
such that every finite patch of tiles present in the tiling appears in some higher order image of some tile. Such patches are
 called {\em legal}. In practice, we generate tilings by iterating the substitution from an initial patch of
tiles, positioning the obtained patches so that each generation properly contains the previous one in its interior,
and take the limit of the process. If the initial patch is such that it appears in a higher order image
of some tile then the obtained limit is clearly a substitution tiling.

A tiling is  {\em recurrent\/} if every finite patch of tiles that appears in the tiling appears infinitely many times in it.
It is {\em uniformly recurrent} (or {\em quasiperiodic}, or {\em repetitive\/}~\cite{baake2013aperiodic})
if every finite patch that appears somewhere in the tiling appears within distance $D$ of every point of the plane, for some $D$. The value $D$ may depend
on the patch. If $D$ can be bounded by a linear function of the diameter of the patch then the tiling is {\em linearly recurrent},
or {\em linearly repetitive\/}~\cite{baake2013aperiodic}. Tilings generated by primitive substitutions are
automatically uniformly recurrent. Indeed, any finite patch that appears in a tiling is legal, i.e., appears in some higher order image of some tile, and by primitivity then, in the $m$'th order
image of every tile, for some $m$. Radius $D$ can be taken to be the maximum diameter of the $m$'th order images of the tiles. A more careful analysis reveals that
our tilings are even linearly recurrent.

In this paper we are considering the problem of finding a primitive substitution that generates a tiling with $n$-fold rotational symmetry. Values $n=2,3,4$ and $6$ are trivial, and there are
even periodic solutions: The regular square tiling for $n=4$ and the case $n=6$ are considered below as the first member of the {\sc Sub Rosa} family.
The famous Penrose rhombuses provide a solution to the case $n=5$~\cite{Gardner,Grunbaum,Penrose}, and the Ammann-Beenker tiling for $n=8$ \cite{Beenker}. Recently, in~\cite{maloney} a computer algorithm was described to search
for solutions for arbitrary $n$. Many substitution rules were discovered for $n=5$ and $n=7$, but the computational complexity of the problem was reported to be prohibitive already in the case $n=11$.

We are not aware of known solutions for general $n$.
In~\cite{harriss} primitive rhombic substitutions were provided for all $n$
that can be iterated on an $n$-fold rotationally symmetric initial patch to obtain in the limit a tiling with $n$-fold symmetry. However, the initial patch does not appear in
any higher order image of any tile, so the obtained tiling is ``singular'' and not a substitution tiling according to the stricter definition of the present paper.
Moreover, the tiling is not recurrent, since the initial patch only appears once in the tiling. In~\cite{harriss} the term \emph{non-singular substitution tiling}
was used to describe the type of substitution tilings used in the present paper.

The main result is the following theorem.

\begin{theorem}
\label{maintheorem}
For every $n$, there exists a quasiperiodic rhombic substitution tiling with $2n$-fold rotational symmetry.
\end{theorem}

\noindent
We start by setting the notations in Section~\ref{sec:notations}.
The discussion is then divided in two parts, depending on whether $n$ is odd or even.
In Section~\ref{sec:rules} we show, as examples, our substitutions for small odd $n$,
and discuss the scaling factors and the edge substitution rules for arbitrary odd $n$. Section~\ref{sec:even} provides
the analogous discussion for even $n$. In Section~\ref{sec:general} we provide the main proof.
We explicitly describe the edge substitution rule, and
express the boundary of the enlarged rhombus as a circular word of unit vectors.
In Section~\ref{sec:rewrite} we develop a rewrite system on the boundary word
to check the tileability  of the interior. In Section~\ref{sec:boundaryword} the
notation is set up for the complete case analysis that we do in Section~\ref{sec:caseanalysis}.

\section{Notations}
\label{sec:notations}

We use only rhombic tiles. Let $n\geq 2$ be a fixed integer. For any positive integers $x$ and $y$ that satisfy $x+y=n$, we denote
by pair $(x,y)$ the rhombus with unit length edges, and
angles $\frac{x\pi}{n}$ and $\frac{y\pi}{n}$. Pairs $(x,y)$ and $(y,x)$ are the same shape.
These $\lfloor \frac{n}{2}\rfloor$ \emph{unit rhombuses} are used in our $2n$-fold
symmetric solutions. We label the corners of the $(x,y)$ rhombus by $x$ and $y$. In a valid edge-to-edge tiling then, the labels of the corners meeting at a vertex have to sum up to $2n$.
The patch of tiles that is an image of a rhombus under our substitution is called a \emph{super-rhombus}.



The system described here utilizes a rotational and reflection symmetric simple arrangement of rhombuses around a single point.
We present here two such patterns: The first one has
$n$-fold rotational symmetry, the second one $2n$-fold symmetry.

First we construct the smaller pattern where we assume $n\geq 3$. We place $n$ copies of $(2,n-2)$-rhombuses around a point with their $\frac{2\pi}{n}$ vertices
meeting at the point. This pattern is surrounded by a ring of $n$ copies of $(4,n-4)$-rhombuses. These properly match at the corners since $(n-2)+(n-2)+4$ equals $2n$.
This pattern is in turn surrounded by a ring of $(6,n-6)$ rhombuses, then by $(8,n-8)$ rhombuses, and so on.
This procedure is repeated until the outmost ring of $(n-1,1)$ or $(n-2,2)$ is reached, depending on whether $n$ is odd or even.
On each step, the angles add up properly at each vertex: on the $k$'th round the four meeting angles
have labels $2(k-1)$, $n-2k$, $n-2k$ and $2(k+1)$, and these add up to $2n$. Note that this analysis also holds on the first round $k=1$,
with the interpretation that the process is initialized with $n$ \emph{zero rhombuses}, i.e., mere edges.

The boundary of the obtained pattern is a regular $2n$-polygon with edges of length $1$, if $n$ is odd, and a
regular $n$-polygon with edges of length $2$ when $n$ is even.
The patterns for $n=3,4,\dots , 8$ are shown in Figure \ref{fig:rose1}.

\begin{figure}[htb]
\begin{center}
 \includegraphics[width=0.8\textwidth]{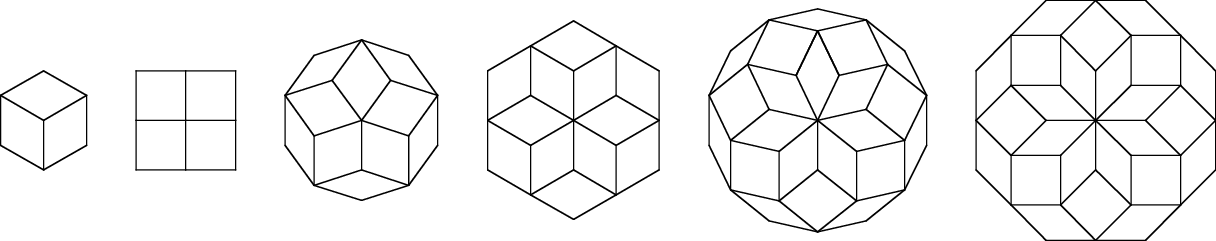}
\end{center}
\caption{Rose $R_1$ for $n=3,4,5,6,7,8$.}
\label{fig:rose1}
\end{figure}

Secondly we construct the larger pattern. Here $n\geq 2$. We place $2n$ copies of
$(1,n-1)$-rhombuses around a point with their $\frac{\pi}{n}$ vertices meeting at the point. This pattern is surrounded by a ring of $2n$
copies of $(2,n-2)$-rhombuses. Which in turn is surrounded by a ring of $2n$ copies of $(3,n-3)$-rhombuses. This procedure is repeated until we construct
the outmost ring using $(n-1,1)$ rhombuses. The result is a regular $2n$-polygon with edges of length $2$. Again, it is straightforward to verify that
the angles add up properly at each vertex. Also here we can imagine the  zeroth round to consist of
 $2n$ zero rhombuses (=unit edges) around the
central point.

Frequently we leave out a number of outer rings. For example, Figure \ref{fig:rose2} shows patterns for $n=2,3,\dots ,7$ where the last ring of $(n-1,1)$-rhombuses is omitted. Note that in the case $n=2$ the omission of the last ring leaves only the degenerate zeroth round.

\begin{figure}[htb]
\begin{center}
 \includegraphics[width=\textwidth]{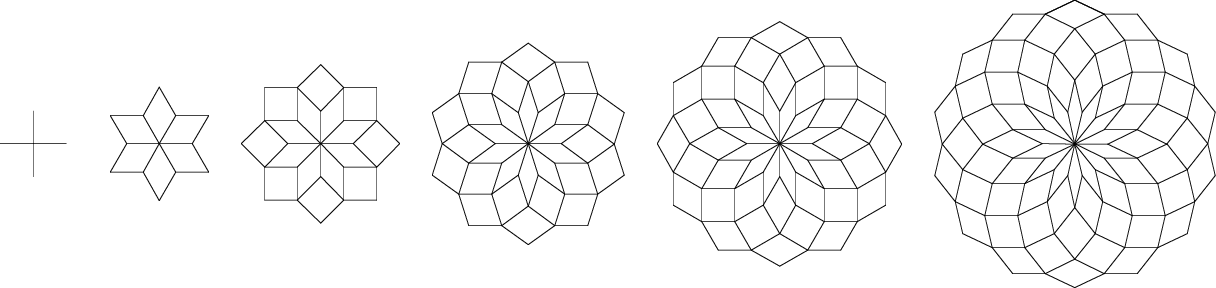}
\end{center}
\caption{Rose $R_2^1$ for $n=2,3,4,5,6,7$.}
\label{fig:rose2}
\end{figure}

These patterns greatly resemble a flower with its petals and therefore we  call them \emph{roses} and denote them by $R_a^b$, where $a\in\{1,2\}$ denotes the type of the rose,
and $b\in\N$ denotes number of missing rings. That is, the patterns shown in Figure~\ref{fig:rose1} are $R_1^0$, or simply $R_1$,
and patterns in Figure~\ref{fig:rose2}  are $R_2^1$. Roses $R_1^b$
have $n$-fold rotational symmetry while roses $R_2^b$ have $2n$-fold rotational symmetry.
Note that $R_1$ for even $n$ only uses tiles $(x,y)$ with even $x$ and $y$.
In fact, $R_2$ for any $n$ is identical to $R_1$ for $2n$.
For $n\geq 6$ it is possible to construct even larger, a third type of rose-pattern with $2n$-fold rotational symmetry. Examples of such larger type are in Figures~\ref{fig:sub7} and \ref{fig:sub12}, and if completed, the result is a regular $2n$-polygon with edges of length 4. However, this, or $R_1$, or any another possible new type of rose is not necessary to prove our main theorem. Instead,
roses $R_2^1$ play a central role in the constructions.

The tilings discussed in this paper use substitutions and rose like patterns, the system is hence called {\sc Sub Rosa}.
We first consider separately the cases of odd $n$ and even $n$, starting with the odd cases.

\section{Scaling factors and substitution rules for odd $n$}
\label{sec:rules}
For $n\in\N\setminus\{0\}$ we define $S(n)$ to be the scaling factor of the substitution. For odd $n$,
\begin{equation}
\label{eq:scaling}
S(n)=\cos(\frac{\pi}{2n})/\sin^2(\frac{\pi}{2n}).
\end{equation}

\subsection{Substitution rule for $n=3$}
\label{subsec:3}


For $n=3$ the {\sc Sub Rosa} tiling is periodic. For the sake of completeness we will first examine the tiling for value $n=3$ before moving to the non-periodic tilings starting with $n=5$.
For $n=3$ there is only one rhombus in the tile set, namely $(1,2)$. The substitution rule is shown on the left in Figure \ref{fig:sub3}. The enlarged rhombus is depicted with thick edges and is replaced by
$12$ rhombuses $(1,2)$. The edges of the enlarged rhombus bisect
two unit rhombuses out of which one is counted in while the other one is counted out.
The resulting tiling after one step of substitution on the starting pattern $R_2^1$ is seen on the right in Figure \ref{fig:sub3}.

\begin{figure}[htb]
 \begin{center}
  \includegraphics[width=0.8\textwidth]{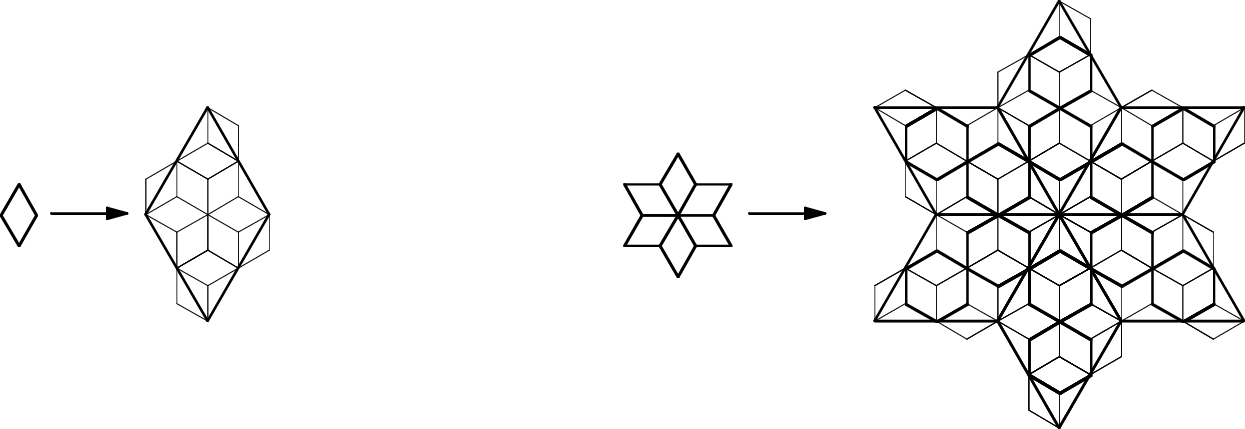}
 \end{center}
\caption{Substitution rule for $n=3$, and the rose $R_2^1$ for $n=3$ after one substitution.}
\label{fig:sub3}
\end{figure}

The following aspects are easy to see in the case $n=3$, but they also hold for the larger values of odd $n$.
\begin{itemize}
\item \emph{The edge substitution rule\/}: All edges of enlarged rhombuses bisect an identical sequence of unit rhombuses.
The sequence has even length and it is mirror symmetric. If we represent each bisected unit rhombus as the label
of its bisected angle, the edge substitution rule can be written down as the sequence $\Sigma(n)$
    of these labels. For example, the edge substitution rule for the case $n=3$ is
    represented by the sequence $\Sigma(3)=1,1$ since the edge of an enlarged rhombus cuts two unit rhombuses along their long diagonals, bisecting their angles of label $1$. The edge substitution rules for larger values of odd $n$ are shown in Table~\ref{tab:composition}.
\item The edge substitution rule above guarantees that the super-rhombuses match each other without gaps or overlaps, so that tile substitutions are consistent. We simply include in the
super-rhombus half of the unit rhombuses that are bisected by its edge. More precisely, when following the edges of the enlarged rhombus clockwise, we count in the first half of the
bisected unit tiles of each edge and count out the second half. In this way each unit rhombus gets included exactly once since each orientation of an edge between two rhombuses is
clockwise for one of the incident rhombuses and counterclockwise for the other.
\item In an edge-to-edge tiling by super-rhombuses there are roses $R_2^1$ centered at all corners of the super-rhombuses.
In other words, a super-rhombus has a sector of $R_2^1$ centered at each corner, and these four sectors together form the full $R_2^1$ rose.
Combined with the fact that there is at least one unit rhombus vertex completely in the interior of each super-rhombus, this implies that the second order image of each tile contains pattern $R_2^1$. Also, since $R_2^1$ contains a copy of every prototile in our system, the patch substituted for each tile contains a copy of every prototile and the substitution is then primitive.
\item We iterate the substitution from the starting pattern $R_2^1$. By the point above, the image of $R_2^1$ has $R_2^1$ at its center,
so we can align the centers of consecutive generations and take the limit to obtain a substitution tiling, {\sc Sub Rosa}.
Strictly speaking, for odd $n$, the central
rose $R_2^1$
quarter turns in each generation. The roses at even generations and odd generations properly align with each other,
but between consecutive generations a quarter turn is required. The generated tiling is a fixed point of the second iterate of the substitution.
\item In all our drawings the depicted unit rhombuses,
in fact, come with an isometry $\varphi$ that maps a rhombic prototile $T$ into the
given position. Because the rhombuses have the dihedral $D_2$ symmetry group ($D_4$ in the case of a square),
the image $\varphi(T)$ alone does not contain the full information about $\varphi$. To identify $\varphi$
uniquely, one may consider the unit rhombuses in our illustrations {\em oriented}.
All tiles are positively oriented, i.e., they come with an even isometry. The orientations in the
starting pattern $R_2^1$ are invariant under the $2n$-fold rotational symmetry of the pattern,
and the sectors of $R_2^1$ at the corners of super-rhombuses are oriented correspondingly to
guarantee the proper orientations in the roses $R_2^1$ that are formed around the vertices of the super-rhombuses.

Note that the orientation of tiles becomes irrelevant if the super-rhombuses have the same symmetries as the
corresponding unit rhombus: in this case the substitution is deterministic even without knowing the orientations.
We argue in Remark~\ref{rem:argue} in Section~\ref{sec:rewrite}
that our substitutions can always be made with the necessary symmetries.
\item We always iterate {\sc Sub Rosa} from the start pattern $R_2^1$.  As patch
$R_2^1$ has $2n$-fold rotational symmetry, the obtained tiling has that symmetry as well.
As the substitution is primitive and the start pattern $R_2^1$ appears in the second order image of each tile, the
tiling is uniformly recurrent.

\end{itemize}

\begin{figure}[htb]
\begin{center}
 \includegraphics[width=0.9\textwidth]{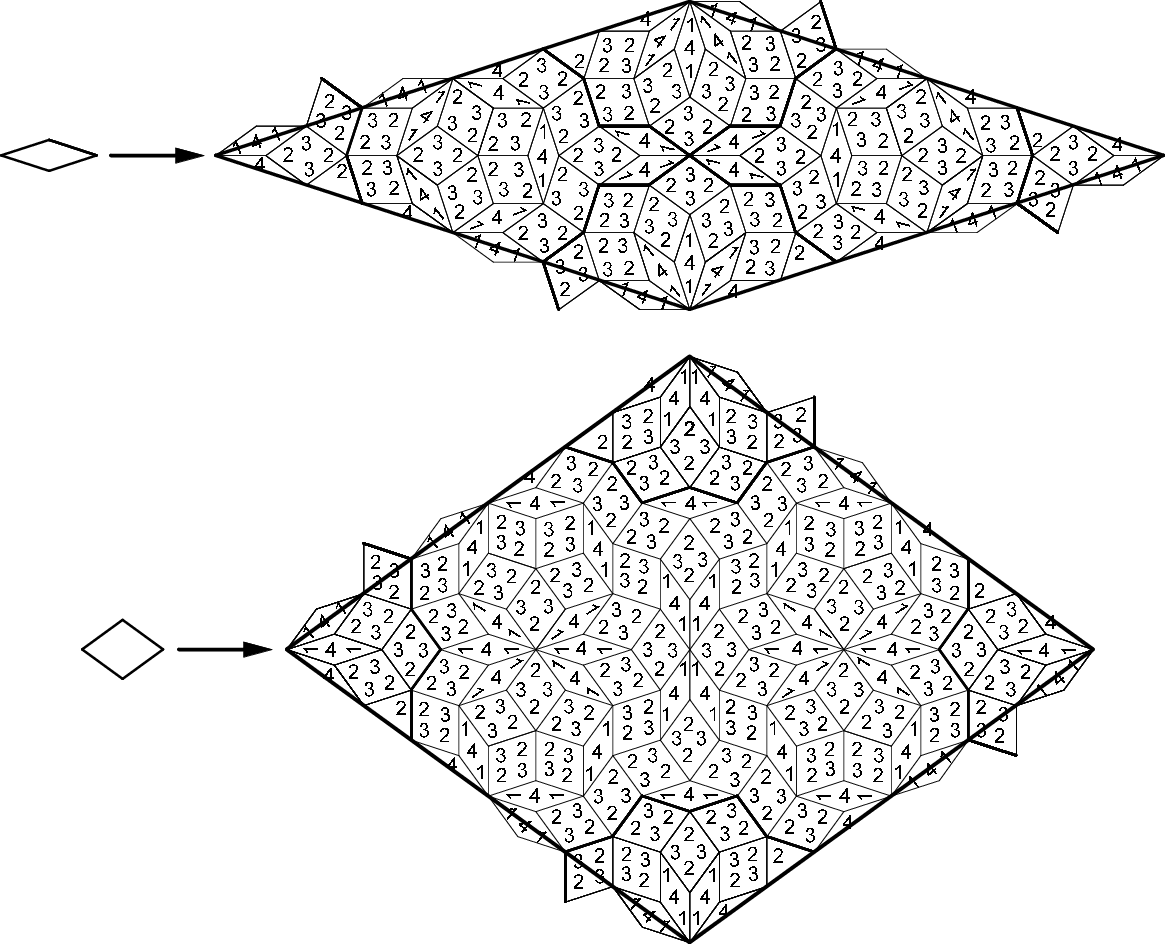}
\end{center}
\caption{Substitution rule for $n=5$.}
\label{fig:sub5}
\end{figure}


\subsection{Substitution rules for $n=5$}
\label{subsec:5}

For $n=5$ the tile set consists of two rhombuses, $(1,4)$ and $(2,3)$. The substitution rule is shown in Figure \ref{fig:sub5}.
In the figure each integer label represents a multiple of $\frac{\pi}{5}$. It is easy to see that each meeting point of vertices sums up to 10 or in other words, is a full circle $2\pi$.
All six points listed above for $n=3$ apply also here, as they do for all odd $n$. Figure \ref{fig:bigrose5} shows the first image of rose $R_2^1$.
Note, again, roses $R_2^1$ centered at all corners of the super-rhombuses.


\begin{figure}[hptb]
\begin{center}
 \includegraphics[width=0.9\textwidth]{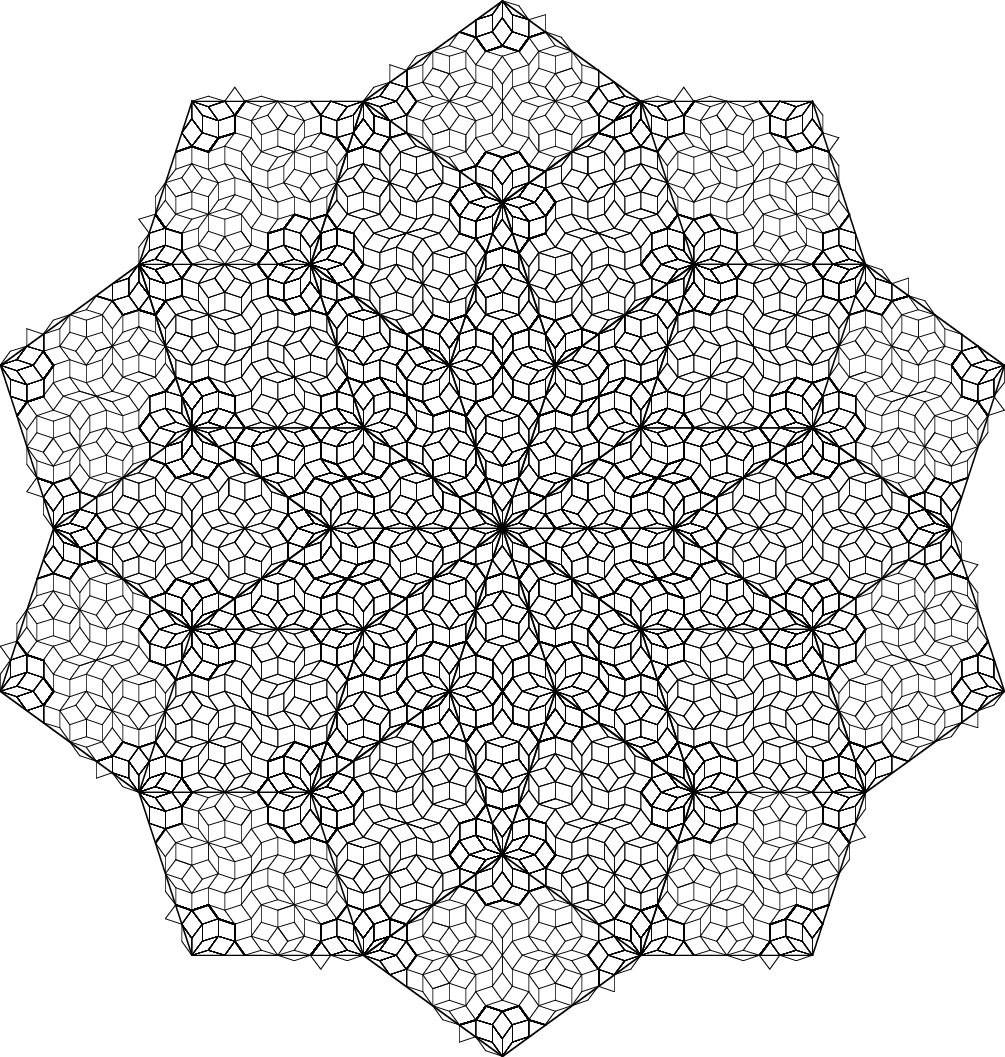}
\end{center}
\caption{Rose $R_2^1$ for $n=5$ after one substitution.}
\label{fig:bigrose5}
\end{figure}

\subsection{Substitution rules for $n=7$}

The substitution rule for $n=7$ is shown in the Figure \ref{fig:sub7}. Similarly to case of $n=5$, it is easy to verify that at each intersection of rhombuses, the sum of the angles is $2\pi$.

\begin{figure}[htb]
\begin{center}
 \includegraphics[width=0.99\textwidth,]{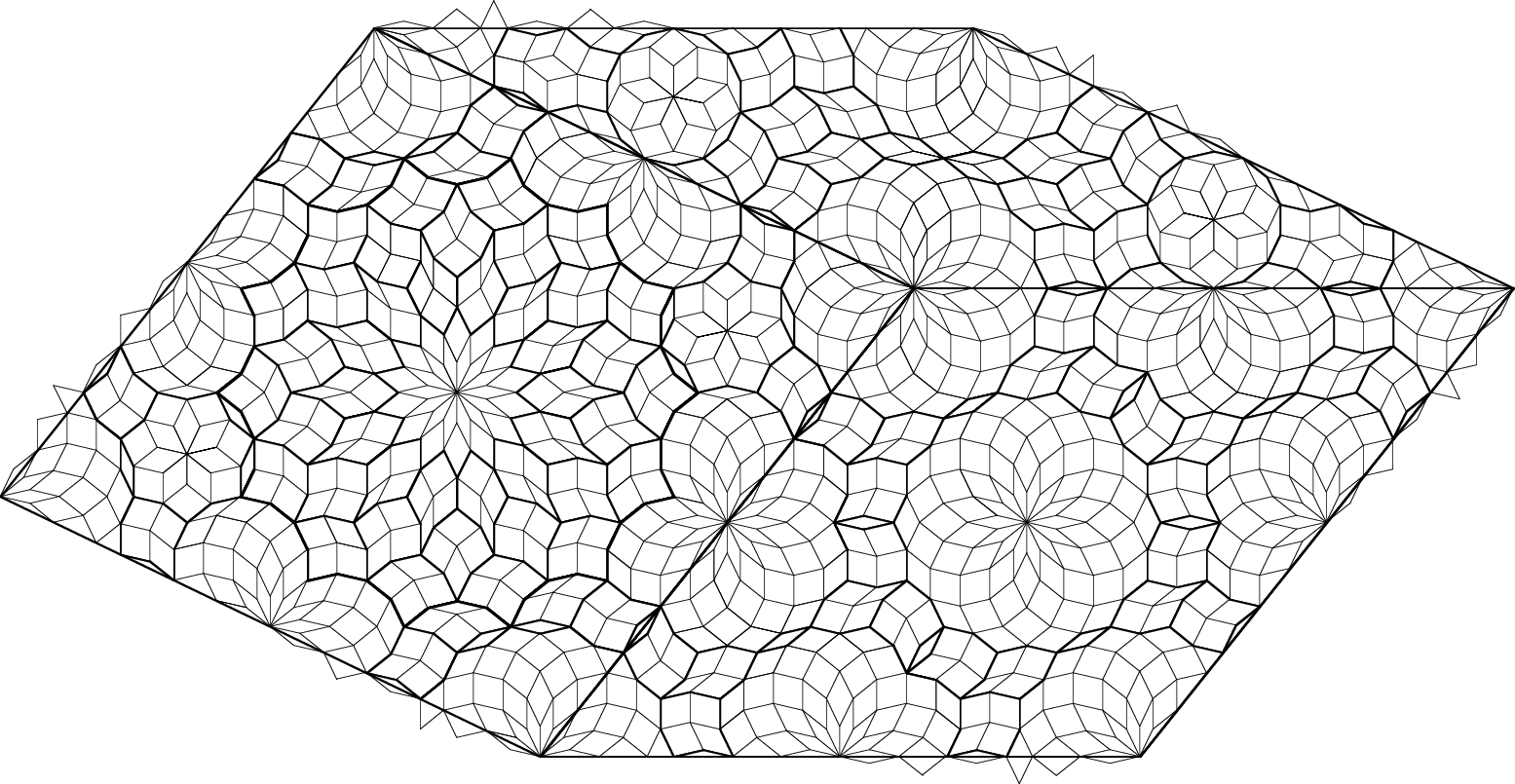}
\end{center}
\caption{Super-rhombuses of substitution rule for $n=7$.}
\label{fig:sub7}
\end{figure}

\subsection{Compositions of super-rhombus' edges}

Every {\sc Sub Rosa} tiling is formed in such a
way that the edge of the super-rhombus bisects all unit rhombuses along it. The length of this edge is $S(n)$, the scaling factor of the substitution.

Recall that we identify the edge substitution rule by the sequence $\Sigma(n)$ of the labels of the angles the edge bisects.
In Table \ref{tab:composition} is depicted the sequence $\Sigma(n)$ for small odd $n$, where symbol $\mid$ denotes the middle point of the edge.
The underlined values in the table represent rhombuses which are inside roses $R^1_2$ centered at the corners of the super-rhombus.
There is a simple rule to form $\Sigma(n)$: The first half of  $\Sigma(n)$ consists of the (underlined) sequence
$1,3,5,\dots ,(n-2)$, followed by the
mirror images of the underlined parts of $\Sigma(3), \Sigma(5), \dots , \Sigma(n-2)$, that is, by $1$, $31$, $531$, etc.
The second half of $\Sigma(n)$ is the mirror image of its first half. This blueprint provides the edge substitution rule for all odd $n$.

\begin{table}[htb]
\caption{Composition of super-rhombus' edges for first $n$ given in unit rhombuses.}
\label{tab:composition}
\begin{tabular}{cc}
\hline\noalign{\smallskip}
$\Sigma(3)$ & $\underline{1}\mid\underline{1}$ \\
$\Sigma(5)$ & \underline{1-3}-$1\mid1$-\underline{3-1} \\
$\Sigma(7)$ & \underline{1-3-5}-1-3-$1\mid1$-3-1-\underline{5-3-1} \\
$\Sigma(9)$ & \underline{1-3-5-7}-1-3-1-5-3-$1\mid1$-3-5-1-3-1-\underline{7-5-3-1} \\
$\Sigma(11)$ & \underline{1-3-5-7-9}-1-3-1-5-3-1-7-5-3-$1\mid
1$-3-5-7-1-3-5-1-3-1-\underline{9-7-5-3-1} \\
\noalign{\smallskip}\hline
\end{tabular}
\end{table}

We choose all the rhombuses on the left side of the midpoint to be counted in
and all rhombuses on the right side of midpoint to be counted out from the super-rhombus. Any other partition would work as well, as long as it is mirror symmetric.

\emph{Diagonal measure} is the length of the diagonal of a unit rhombus $(k,n-k)$ that bisects $k$, and is denoted by $d_n(k)$. The diagonal measure is given by the formula
\begin{equation}
\label{eq:diagonal}
d_n(k)=2\cos(k\frac{\pi}{2n}).
\end{equation}
From $\Sigma(n)$ we can read the scaling
factor $S(n)$ as a sum of diagonal measures $d_n(k)$. For example,
from $\Sigma(7)=1,3,5,1,3,1,1,3,1,5,3,1$ we obtain that
$$
S(7)=2(d_7(1)+d_7(3)+d_7(5)+d_7(1)+d_7(3)+d_7(1)).
$$
From the proposed general structure of $\Sigma(n)$ for all odd $n$, we get
$$S(n)=(n-1)d_n(1)+(n-3)d_n(3)+\ldots+2d_n(n-2).$$
The scaling factor in
(\ref{eq:scaling}) was inferred from this relation to diagonal measures (\ref{eq:diagonal}).

\section{Substitution rules for even $n$}
\label{sec:even}
In this section we consider briefly {\sc Sub Rosa} tilings for even values of $n$. Now the scaling factor is different
\begin{equation}
\label{eq:scale2}
S(n)=
 \frac{2}{1-\cos(\frac{\pi}{n})}.
\end{equation}
The smallest case $n=2$ has scaling factor $2$: it is the square substitution where the square is replaced by four squares. The resulting tiling
is the regular square tiling, which is also the only
edge-to-edge rhombic tiling in this case.

Consider an arbitrary even $n\geq 2$.

\begin{itemize}
\item  \emph{The edge substitution rule\/}: The edges of enlarged rhombuses bisect some unit rhombuses and coincide with the edges of some unit rhombuses. In the latter case we say the
edge bisects a zero rhombus, and indicate such situation by label $0$ in $\Sigma(n)$. The sequence $\Sigma(n)$ is mirror symmetric and of even length.
This means that, as in the odd case,
the super-rhombuses match each other without gaps or overlaps if we count in the super-rhombus the bisected tiles in the first half of $\Sigma(n)$,
and count out the bisected tiles in the second half of $\Sigma(n)$. The edge substitution rules for small even $n$ are shown in Table~\ref{tab:composition2}.
\item The starting pattern is again $R_2^1$. All super-rhombuses have a sector of $R_2^1$ at each vertex, but now the sector is
aligned so that the first label of $\Sigma(n)$ is $0$ rather than $1$. This means that the image of $R_2^1$ contains at its
center $R_2^1$ in its original orientation. In other words, no quarter rotation is needed to match consecutive generations, and the
final tiling is a fixed point of the substitution.
\item The general structure of
$\Sigma(n)$ for even $n$ is analogous to the odd case:
The first half of $\Sigma(n)$ consists of the (underlined) sequence
$0,2,4,\dots ,(n-2)$, followed by  the
mirror images of the underlined parts of $\Sigma(2)$, $\Sigma(4), \dots , \Sigma(n-2)$, that is, by $0$, $20$, $420$, etc. The second half of
 $\Sigma(n)$ is the reversal of the first half.
The obtained word is a palindrome of even length.
\end{itemize}

\begin{table}[htb]
\caption{Composition of super-rhombus' edges for first even $n$.}
\label{tab:composition2}
\begin{tabular}{cc}
\hline\noalign{\smallskip}
$\Sigma(2)$ & \underline{0}$\mid$\underline{0} \\
$\Sigma(4)$ & \underline{0-2}-0$\mid$0-\underline{2-0} \\
$\Sigma(6)$ & \underline{0-2-4}-0-2-0$\mid$0-2-0-\underline{4-2-0} \\
$\Sigma(8)$ & \underline{0-2-4-6}-0-2-0-4-2-0$\mid$0-2-4-0-2-0-\underline{6-4-2-0} \\
$\Sigma(10)$ & \underline{0-2-4-6-8}-0-2-0-4-2-0-6-4-2-0$\mid$0-2
4-6-0-2-4-0-2-0-\underline{8-6-4-2-0} \\
$\Sigma(12)$ & \underline{0-2-4-6-8-10}-0-2-0-4-2-0-6-4-2-0-8-6-4-2-0$\mid$0-2-4-6-8-0-2-4-6-0-2-4-0-2-0-\underline{10-8-6-4-2-0} \\
\noalign{\smallskip}\hline
\end{tabular}
\end{table}

Rose $R_2^1$ for $n=4$ after one substitution can be seen in Figure \ref{fig:bigrose8}.
For $n=6$ the substitution rule is shown in the Figure \ref{fig:sub12}.

\begin{figure}[htb]
\begin{center}
 \includegraphics[width=0.8\textwidth]{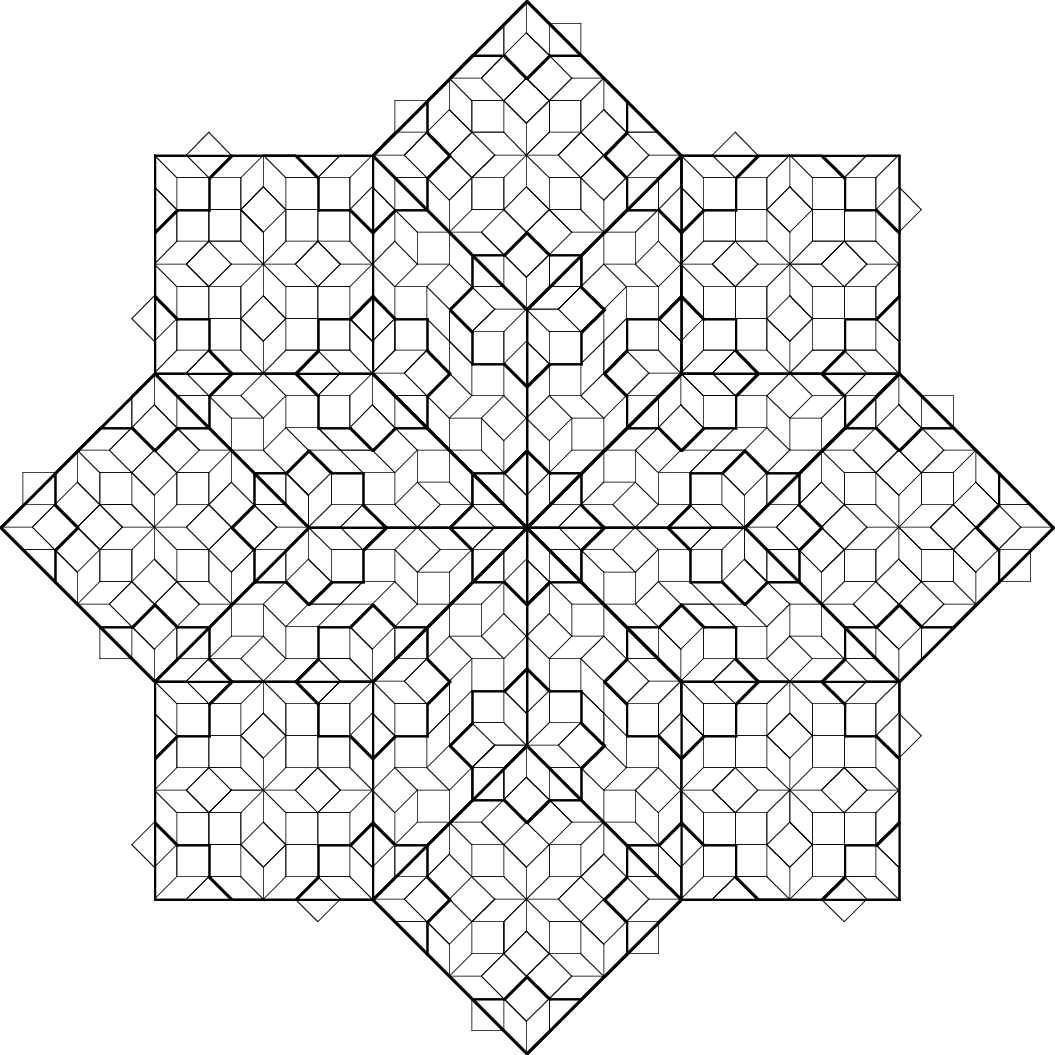}
\end{center}
\caption{Rose $R_2^1$ for $n=4$ after one substitution.}
\label{fig:bigrose8}
\end{figure}

\begin{figure}[htb]
\begin{center}
\includegraphics[width=0.99\textwidth]{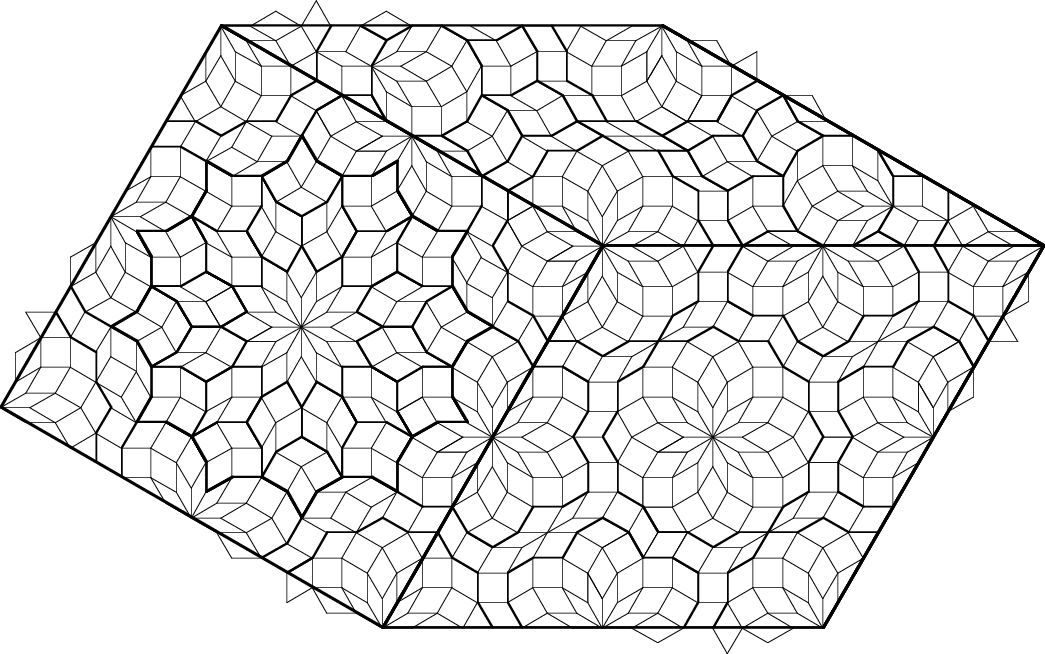}
\end{center}
\caption{Super-rhombuses of the substitution rule for $n=6$.}
\label{fig:sub12}
\end{figure}

\section{General case}
\label{sec:general}

We have shown substitution rules for small values of $n$. In the following we demonstrate that analogous substitutions
exist for all values of $n$, thus proving Theorem~\ref{maintheorem}.
We use the explicit description of super-rhombuses' boundaries as in
Tables~\ref{tab:composition} and \ref{tab:composition2}. Then we use the method in~\cite{kannan,kenyon}
to prove that the interior can be tiled with the unit rhombuses.

We first set our notations. Let $n$ be fixed. All angles will be expressed in units $\frac{\pi}{n}$. Number $2n$ is the full
circle so that angles are considered modulo $2n$.
Direction $0$ is drawn horizontally to the right, so that directions $\frac{n}{2}$, $n$ and $-\frac{n}{2}$ refer to up, left and down, respectively.
For each direction $x$, the antiparallel direction $x+n$ is denoted  by $\anti{x}$. Hence  $\anti{\anti{x}}=x$.

In Tables~\ref{tab:composition} and \ref{tab:composition2}
the structure of super-rhombuses' edges is given in terms of the unit rhombuses
along edges. In the following it is more convenient to consider the sequence of unit vectors that enclose the
interior that needs to be tiled. Each unit vector is represented by its direction, so the \emph{boundary word} of
the super-rhombus is the sequence of directions of the unit vectors on the boundary.
This sequence is a cyclic word so that conjugate words
$uv$ and $vu$ denote the same boundary, just with a different starting point.
We read the boundary word counterclockwise. The unit rhombuses bisected by the super-edge are not included in
interior to be tiled. This means that a unit rhombus $(a,n-a)$ on a super-edge of direction $k$
contributes in the boundary word two unit vectors in directions $k+\frac{a}{2}$ and $k-\frac{a}{2}$, in this order.

In fact, we want to show that the interior can be tiled in such a way that sectors of rose  $R_2^1$ appear centered at all
four corners of the super-rhombus. The underlined symbols in
Tables~\ref{tab:composition} and \ref{tab:composition2} are inside these roses so they get replaced by the boundary
of the rose sector. As an example,
Figure~\ref{fig:boundary} shows the boundary of the region to be tiled in the case of the $(2,3)$ super-rhombus. Note how the boundary traces the sectors
of the $R_2^1$ roses at the corners.

\begin{figure}[hptb]
\begin{center}
 \includegraphics[width=0.6\textwidth]{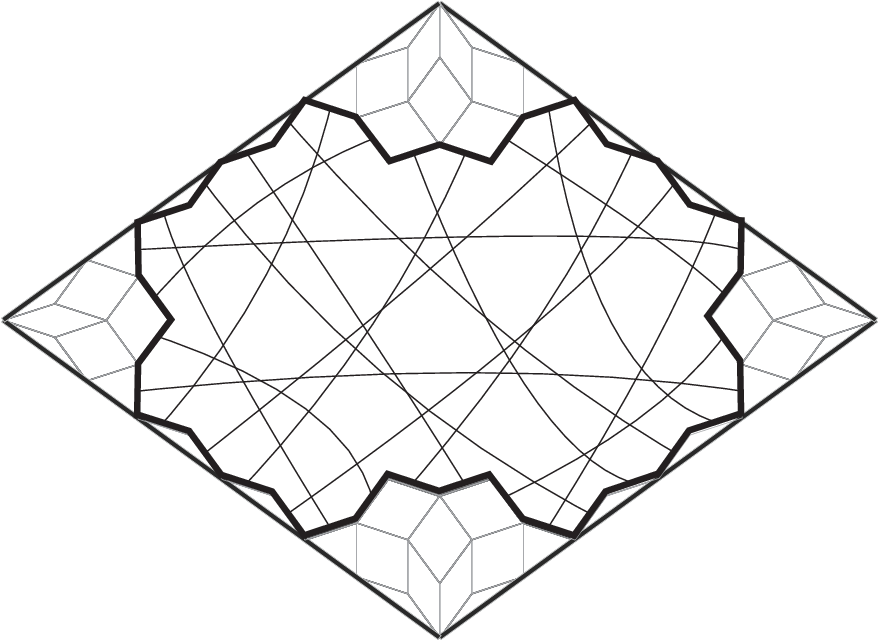}
 \end{center}
\caption{The boundary of the region
to be tiled with unit rhombuses, in the case $n=5$ and super-rhombus $(2,3)$. Lines connect the matching
pairs of unit vectors along the boundary. Each crossing of the lines corresponds to a unit rhombus in the interior. In this case, each crossing provides a properly oriented rhombus which means that the crossing condition is satisfied.}
\label{fig:boundary}
\end{figure}

We use the standard notations on words.
If $u$ and $v$ are two words then $uv$ is their concatenation. A concatenation of $k$ copies of word $u$
is denoted as $u^k$. Though individual letters $a$ are directions, and hence numbers modulo $n$, the power
notation $a^k$ represents the word $aa\dots a$ of length $k$ (rather than number $a$ to power $k$). We denote by $u^R$ the \emph{reversal} of word $u$,
that is, the word obtained by writing the letters of $u$ in the reverse order. Word $u$ is a \emph{palindrome} if $u^R=u$.
The empty word is denoted as $\varepsilon$. It has length zero.
Sometimes, for clarity, we write words with commas separating the letters. So $aabab$ and $a,a,b,a,b$ denote the same word of length five.

For any direction $x$, we define the operation $\sigma_x$ on words that increments each letter by constant $x$.
This corresponds to turning the entire path by angle $x$. In particular, $\sigma_n$ orients the path in the
opposite direction. We extend the notation of antiparallelism to words so that $\anti{u}=\sigma_n(u)$
denotes the path antiparallel to $u$, that is, half turned $u$.

Since the same palindromic edge substitution is used on all four
edges of a super-rhombus, it easily follows that the contributions
on the boundary word
by opposite super-edges  are antiparallel to each other, and the same holds for the contributions by
the rose sectors at opposite corners of the super-rhombus. Hence the boundary word is of the form $u\anti{u}$
where $u$ is the word through half of the boundary.
It follows from this symmetry that,
in each direction $a$, the boundary contains equally many unit
vectors in directions $a$ and $\anti{a}$. This is the {\em balance condition\/} in~\cite{kannan}.

It is also easy to see that the boundary does not cross itself.  Only in the case of
the $(1,n-1)$ super-rhombus for odd $n$  the roses at the opposite $n-1$ corners touch in the middle, but without
crossing each other. In the terminology of~\cite{kannan} the  boundary is {\em simple}.

Based on~\cite{kannan}, to prove that the interior can
be tiled with the unit rhombuses it only remains to show
that the boundary vectors can be {\em matched\/}
in an appropriate way. Each letter $a$ on the boundary word must be matched with
an occurrence of the antiparallel letter $\anti{a}$. A \emph{crossing} is formed by two matched pairs $a \frown \anti{a}$
and $b \frown \anti{b}$ if they occur in the circular boundary word in the interleaved order
$\dots a \dots b\dots \anti{a}\dots \anti{b}\dots$. For each such crossing, it is required that
the path $ab\anti{a}\anti{b}$ forms a rhombus in the counterclockwise direction, that is,
$a<b<a+n \pmod{2n}$.
This is the \emph{crossing condition}.

The convex crossing condition used in~\cite{kannan} is slightly weaker as it also allows
crossings between matched pairs $a \frown \anti{a}$ and $a \frown \anti{a}$ with the
same labels. We prefer to forbid this because we then get unique matchings:
In our setup, all occurrences of $a$ and $\anti{a}$ are in the
opposite halves of the boundary word.
This implies that there is a unique way of matching the occurrences of
$a$ and $\anti{a}$ with each other without crossings. Indeed, the $i$'th
occurrence of $a$ must be paired with the $i$'th last occurrence of
$\anti{a}$.

What needs to be checked is that this unique matching respects the crossing condition
for all distinct directions $a$ and $b$. This in mind, for any $a$ and $b$
we define the \emph{projection} function $\pi_{a,b}$ that erases from a boundary word
all letters except $a,b,\anti{a}$ and $\anti{b}$. The crossing condition of the boundary word $u$
can be expressed
equivalently as the requirement that all projections $\pi_{a,b}(u)$ satisfy the crossing condition.
This turns out to be equivalent in our setup to the property that $\pi_{a,b}(u)$ defines a non-crossing
cycle in the counterclockwise direction.

\begin{example}
\label{ex:boundary23}
Consider the boundary word of the $(2,3)$ rhombus in Figure~\ref{fig:boundary}. Keeping the orientation of the figure, the directions of the unit vectors are from $\Z+\frac{1}{2}$.
Reading counterclockwise, starting where the leftmost rose segment ends, the boundary word is
$$
\begin{array}{l}
\frac{-1}{2}, \frac{-3}{2},\frac{-1}{2}, \frac{-3}{2}\ \vline\
\frac{1}{2}, \frac{3}{2},\frac{-1}{2}, \frac{1}{2}, \frac{-3}{2}, \frac{-1}{2}\ \vline\
\frac{3}{2}, \frac{1}{2},\frac{3}{2}, \frac{1}{2}\ \vline\
\frac{5}{2}, \frac{7}{2},\frac{3}{2}, \frac{5}{2}\ \vline\ \\ \\
\hspace*{5mm}
\frac{9}{2}, \frac{7}{2},\frac{9}{2}, \frac{7}{2}\ \vline\
\frac{-9}{2}, \frac{-7}{2},\frac{9}{2}, \frac{-9}{2}, \frac{7}{2}, \frac{9}{2}\ \vline\
\frac{-7}{2}, \frac{-9}{2},\frac{-7}{2}, \frac{-9}{2}\ \vline\
\frac{-5}{2}, \frac{-3}{2},\frac{-7}{2}, \frac{-5}{2}\ \vline\ \\
\end{array}
$$
Vertical lines indicate the changes between the four rose segments and the four edge segments of the boundary.
There are 5 different directions (when antiparallel directions are not counted separately), so there are ${5 \choose 2} =10$
different pairs of directions that define projections $\pi_{a,b}$.
For example, for $a=\frac{3}{2}$ and $b=\frac{5}{2}$, keeping in mind that
$\anti{a}=\frac{-7}{2}$ and $\anti{b}=\frac{-5}{2}$,
the projection by $\pi_{a,b}$ is
$$
a\ a\ a\ b\ a\ b\ \anti{a}\ \anti{a}\ \anti{a}\ \anti{b}\ \anti{a}\ \anti{b}.
$$
This word defines a non-crossing boundary in the counterclockwise direction. Analogously we can determine the projections
for all pairs $a$ and $b$. The corresponding paths are shown in
Figure~\ref{fig:projections}. The crossing condition is satisfied.
\qed

\begin{figure}[hptb]
\begin{center}
 \includegraphics[width=0.95\textwidth]{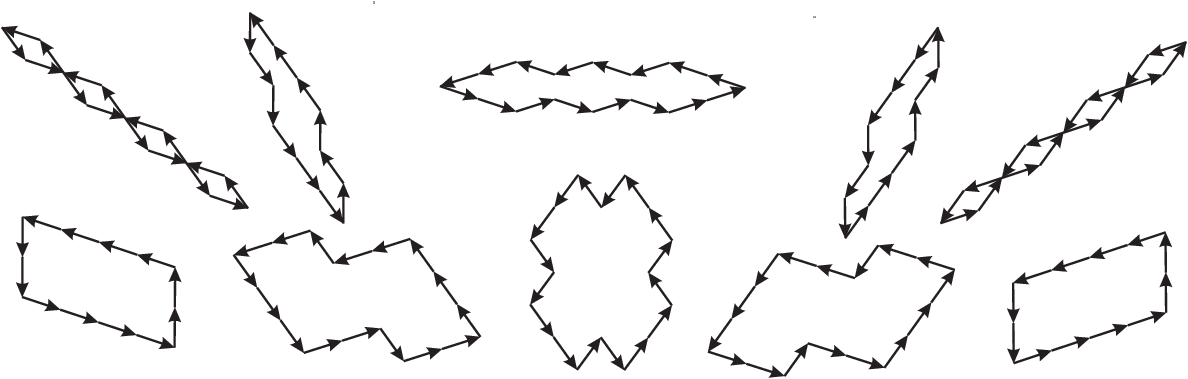}
 \end{center}
\caption{All projections in pairs of distinct directions of the boundary of the $(2,3)$ super-rhombus. All paths are nonintersecting
and counterclockwise, so the crossing condition is satisfied.}
\label{fig:projections}
\end{figure}

\end{example}

The results in~\cite{kannan,kenyon} guarantee that a region surrounded by a simple boundary
that satisfies the balance condition and whose (unique) matching satisfies the crossing condition
has a tiling by parallelograms (Theorem 2 in~\cite{kannan}). In our setup the only non-trivial aspect to check is the crossing condition. In the following section we develop a convenient rewrite system to check this condition.

\begin{remark}
While~\cite{kannan} guarantees a tiling by parallelograms, it is easy to see that the proof in \cite{kannan} provides a tiling by unit rhombuses if the boundary consists of unit length segments.
\end{remark}

\subsection{Rewrite system to check the crossing condition}
\label{sec:rewrite}

Our standing assumption in the following is that $a$ and $b$ are directions such that $a<b<a+n \pmod{2n}$ so that
$ab\anti{a}\anti{b}$ is a proper rhombus in the counterclockwise direction. Let $u,v$ be words over the alphabet
$\{a,b,\anti{a},\anti{b}\}$.
Consider the following eight \emph{rewrite} rules:
\begin{equation}
\label{eq:rewrite}
\begin{array}{rclcrclcrclcrclc}
ab & \rwrule & ba, &\hspace*{5mm}&  b\anti{a} &\rwrule & \anti{a}b, &\hspace*{5mm} & \anti{a}\anti{b} &\rwrule &
\anti{b}\anti{a}, &\hspace*{5mm} & \anti{b}a &\rwrule & a\anti{b},\\
a\anti{a} &\rwrule &\varepsilon, && \anti{a}a & \rwrule & \varepsilon, &&
b\anti{b} &\rwrule &\varepsilon, && \anti{b}b & \rwrule & \varepsilon,
\end{array}
\end{equation}
An application of rule $x\rwrule y$ on word $u$ means that we replace an occurrence of subword $x$ in $u$ by $y$.
More precisely, if $u=u_1xu_2$ and $v=v_1yv_2$ and $x\rwrule y$ is a rewrite rule then
$u$ derives $v$ and we write $u\rewrite v$.
We denote the transitive, reflexive closure of $\rewrite$ by $\rewrite^*$, so that $u\rewrite^* v$ means that $u$ can be turned into $v$
by a sequence of rewrites. Iterating rule $ab \rwrule  ba$, for example, allows us to move $b$'s
any number of positions to the left on a word containing only $a$'s and $b$'s, or to move
$a$'s to the right on such a word. In other words, $au \rewrite^* ua$ and $ub \rewrite^* bu$ for any word $u$ that only contains letters $a$ and $b$.

The top four rewrite rules in (\ref{eq:rewrite}) correspond to snapping
two consecutive edges of the rhombus $ab\anti{a}\anti{b}$ in reverse order.
The last four rules allow eliminating a letter and its reversal that are  next to each other.

Suppose $u\rewrite v$ using the first rewrite rule $ab\rwrule ba$. If $v$ satisfies the crossing
condition, so does $u$ with the same matchings. Indeed, the
only new crossing in $u$, not present in $v$,
connects pairs
$a \frown \anti{a}$ and $b \frown \anti{b}$
of symbols that appear in $u$ in the (circular)
order $\dots ab\dots \anti{a}\dots \anti{b}\dots$. Such order is allowed by the crossing condition.
The same argument applies to all four rewrite rules that reverse the order of two symbols.

Suppose then that  $u\rewrite v$ using the eliminating rewrite rule $a\anti{a}\rwrule \varepsilon$.
If $v$ satisfies the crossing
condition, so does $u$ when we connect the two letters in the eliminated pair $a\anti{a}$ with each other,
and match all other letters in the same way as they were matched in $v$.
In this way $u$ has exactly the same crossings as $v$.
The same argument applies to any other eliminating rewrite rule.

We have seen that if $u\rewrite^* v$ and if $v$ satisfies the crossing condition then also $u$ satisfies the crossing condition.
In particular, $u\rewrite^* \varepsilon$ guarantees that $u$ satisfies the crossing condition. In fact, it is not difficult to see that this condition is
also necessary:

\begin{lemma}
\label{lem:firstlemma}
Let $u$ be a word over the alphabet $\{a,b,\anti{a},\anti{b}\}$. Then $u$ has a matching that satisfies the crossing condition if and only if
$u\rewrite^* \varepsilon$.
\end{lemma}

\begin{proof}
We have seen above that $u\rewrite^* \varepsilon$ implies that there is matching in $u$ that satisfies the crossing condition.
Let us prove the converse direction. Define a partial order among words with matchings: $u<v$ if $u$ is shorter than $v$, or if
$u$ and $v$ have the same lengths but $u$ has fewer crossings than $v$.

Let $u$ be a non-empty word with a matching that satisfies the crossing condition. We want to prove that $u\rewrite v$ for some $v<u$
that also satisfies the matching condition.

Assume first that $u$ contains two consecutive letters that are matched with each other. Hence the letters are
antiparallel to each other, and can be eliminated by a rewrite rule. This produces a shorter word $v$, so $v<u$. The remaining
letters in $v$ can be matched exactly as in $u$, satisfying the crossing condition.


Assume next that $u$ contains two consecutive letters $xy$ whose connections cross each other. By the crossing condition, $xy\anti{x}\anti{y}$
must be a proper rhombus in the counterclockwise direction, so $xy\rwrule yx$ is a valid rewrite rule. When we apply it to $u$, and keep
letter matchings unchanged,
we obtain $v$ that has the same length but one less crossing than $u$. Hence $u\rewrite v$ and $v<u$. Clearly $v$ satisfies the crossing condition since we obtained it from $u$ by only untangling one crossing.

Finally, assume that $u$ has no consecutive letters that are connected to each other or whose connections would cross each other.
Let $x$ and $y$ be two letters in $u$  that are either connected to each other or whose connections cross, and assume their (cyclic) distance $d$
is as short as possible. The letters are not next to each other so there is a letter $z$ between them. But $z$ cannot be connected to another letter
$\anti{z}$ between $x$ and $y$ because then the distance from $z$ to $\anti{z}$ would be shorter than $d$, contradicting the choice of $x$ and $y$.
Hence the connection of $z$ necessarily crosses the $x, y$ pair, and hence it crosses either the connection of $x$ or the connection of $y$.
But the distance from $z$ to both $x$ and $y$ is shorter than $d$, which again contradicts the minimality of $d$.

We have proved the claim that  $u\rewrite v$ for some $v<u$
that also satisfies the matching condition. Now, if $v\neq\varepsilon$ the argument can inductively applied on $v$.
By iterating the argument we obtain a sequence $u \rewrite v \rewrite \dots$. The partial order $<$ does not admit infinite decreasing chains
so $\varepsilon$ must be eventually reached.
\qed

\end{proof}

\begin{remark}
\label{rem:argue}

In our setup the boundary $u$ of the tileable
region has all the symmetries of the enlarged rhombus: dihedral group $D_4$ or $D_2$
if the rhombus is or is not a square, respectively. The unique matching also respects these symmetries, and so do all projections $\pi_{a,b}(u)$. Shrinking the remaining tileable
region by adding a unit rhombus tile with two edges on the boundary corresponds to
the application of a rewrite rule on some $\pi_{a,b}(u)$.
The analogous rewrite can be done on
all symmetric positions, thus reducing the tileable region while keeping its symmetries.
By iterating this process we obtain a tiling of the interior of the enlarged rhombus that has all the symmetries of the initial rhombus.

A detail to observe in this process is that
each symmetry of the boundary takes any pair of consecutive edges to a disjoint pair, or to the identical pair of consecutive edges,
but never to a pair that shares exactly one edge with the original pair. For this reason
all symmetric positions can be rewritten independently of each other. In contrast, the reader may consider, for example, tiling
the regular octagon of unit sides: the reflection symmetries prevent a fully symmetric tiling of the interior by unit rhombuses.
\end{remark}

Notice that  $u\rewrite^* v$ implies that $\anti{u}\rewrite^* \anti{v}$. This is because for each rewrite rule
$x\rwrule y$ there is also the rewrite
rule $\anti{x}\rwrule\anti{y}$ in our toolbox (\ref{eq:rewrite}).
This also implies that if $u\rewrite^* v$ and $u\rewrite^* v^R$ for some word $v$
then $u\anti{u}$ satisfies the crossing condition. Indeed,
$u\anti{u} \rewrite^* v^R\anti{v} \rewrite^*\varepsilon$,
where the last steps use elimination rewrites.
In particular, if a palindrome can be derived from $u$ then $u\anti{u}$ satisfies the crossing condition.

\begin{example}
Let $u=aaabab$ so that $u\anti{u}$ is
the boundary word from Example \ref{ex:boundary23}. Since
$aaabab\rewrite^* baaaab$ we have that
$u\anti{u}\rewrite^* baaaab\ \anti{b}\anti{a}\anti{a}\anti{a}\anti{a}\anti{b}\rewrite^* \varepsilon$.
Analogously, all ten projections shown in Figure~\ref{fig:projections} can be reduced
into words of the form $p\anti{p}$ for some palindromes $p$, and hence into the empty word $\varepsilon$.
\qed
\end{example}

We are interested in the matching condition on boundary words that are of the type $u\anti{u}$, by performing rewrites on the half word $u$.
As we work with the half boundary, there is yet another useful operation on the words:
If $u=xy$ is a concatenation of two words then we may swap the
order of $x$ and $y$ while changing $x$ into $\anti{x}$. We write $xy\conjugate y\anti{x}$. When such operation is performed on both halves
of $u\anti{u}$, the (circular) word remains unchanged, that is, if $u\conjugate v$ then $u\anti{u}$ and $v\anti{v}$ are the same (circular) words.
Indeed, $u\anti{u}=xy\anti{x}\anti{y}$ and $v\anti{v}=y\anti{x} \anti{y}x$
for $v=y\anti{x}$.

We denote $u\derive v$ if $u\rewrite v$ or $u\conjugate v$, and by $\derive^*$ we denote the reflexive transitive closure of $\derive$.

\begin{lemma}
\label{lem:lemma2}
If $u\derive^* v\rewrite^* v^R$ for some $v$ then $u\anti{u}$ satisfies the crossing condition.
\end{lemma}
\begin{proof}
As $v\anti{v}\rewrite^*v^R\anti{v}\rewrite^* \varepsilon$, by Lemma~\ref{lem:firstlemma} we know that $v\anti{v}$ satisfies the crossing condition.
For any $x\derive y$, if $y\anti{y}$ satisfies the crossing condition so does $x\anti{x}$. Namely, if $x\rewrite y$ then $x\anti{x}\rewrite y\anti{y}$,
and if $x\conjugate y$ then $x\anti{x}$ and $y\anti{y}$ are the same circular word. Hence, $u\derive^* v$ implies that
$u\anti{u}$ satisfies the crossing condition because $v\anti{v}$ does.
\qed
\end{proof}

All our proofs of tileability use Lemma~\ref{lem:lemma2}. We start with a half $u$ of the boundary word $u\anti{u}$, and derive from $u$
a new word $v$ using operations $\rewrite$ and $\conjugate$. Then we operate on $v$ using only rewrite operations $\rewrite$ to
reach its reversal $v^R$. Of course we could always work on the full boundary $u\anti{u}$ but working on $u$ reduces the size of
expressions and the number of operations by half. We frequently use
$$
xyx^R \derive^* yx^R\anti{x} \derive^* y,
$$
and simply cancel a prefix and a suffix from $u$ if they are reversals of each other.

The following particular case will be used multiple times, so we state it as a separate lemma.

\begin{lemma}
\label{lem:lemma3}
Assume our usual hypothesis that $ab\anti{a}\anti{b}$ is a rhombus in the counterclockwise orientation.
If $u=(ba)^i(b\anti{a})^j$ where $j\geq 1$ then
$u\derive^* v\rewrite^* v^R$ for some $v$. In particular, $u\anti{u}$ satisfies the crossing condition.
\end{lemma}
\begin{proof}
We apply the rewrite rules in (\ref{eq:rewrite}) that move $a$ to the right and $\anti{a}$ to the left over $b$'s,
and cancel $a$ and $\anti{a}$  as they meet.

If $i< j$ then
$$
(ba)^i\ (b\anti{a})^j \derive^* b^ia^i\ \anti{a}^ib^i\ (b\anti{a})^{j-i}
\derive^* b^{2i}\ (b\anti{a})^{j-i} \rewrite^* (\anti{a}b)^{j-i}\ b^{2i},
$$
and the last two words are reverses of each other. In the  last steps we moved $2i+1$ letters $b$ from the beginning
of the word to the end.
If $i\geq j\geq 1$,
$$
(ba)^i(b\anti{a})^j \derive^* (ba)^{i-j}\ b^{j}a^{j}\ \anti{a}^{j}b^j  \derive^* (ba)^{i-j}b^{2j} \rewrite^* b^{2j}(ab)^{i-j},
$$
where the last steps of the derivation
move $2j-1$ copies of $b$ from the end of the word to the beginning.
Also here the last words are reverses of each other.
Lemma~\ref{lem:lemma2} confirms that  $u\anti{u}$ satisfies the crossing condition.

\qed
\end{proof}

\subsection{Boundary words of super-rhombuses}
\label{sec:boundaryword}

\begin{figure}[hptb]
\begin{center}
 \includegraphics[width=0.5\textwidth]{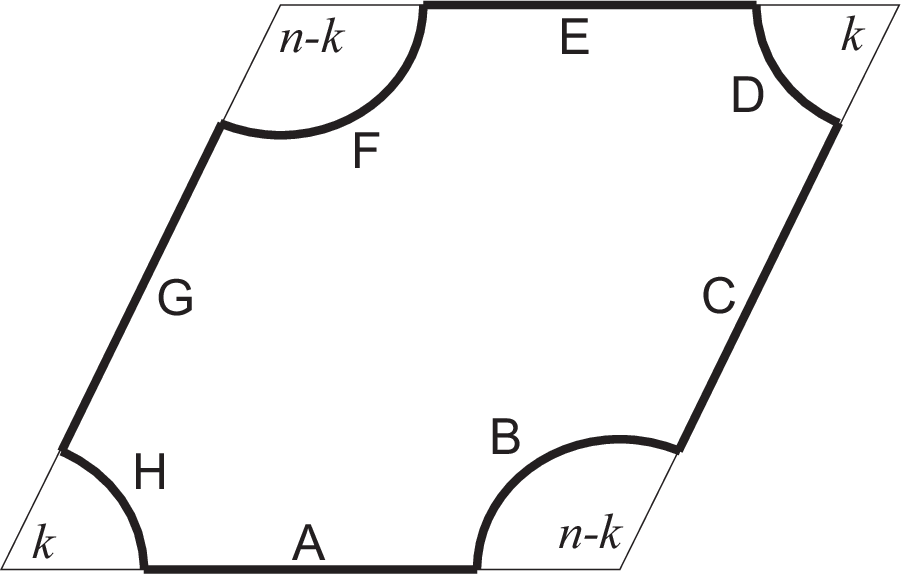}
 \end{center}
\caption{The eight segments of the boundary.}
\label{fig:boundarysegments}
\end{figure}

Let us fix
$k\in\{1,2,\dots,n-1\}$ and consider the super-rhombus $S$ of type $(k,n-k)$. Let us orient $S$ on the plane with its bottom edge horizontal and corner $k$ at the left end
of the base. Super-rhombuses $(k,n-k)$ and $(n-k,k)$ are identical, so it is sufficient to consider
only one of them for each $k$.
We name the rose and edge segments of the boundary of $S$ by letters $A$ through $H$, as shown in Figure~\ref{fig:boundarysegments}.
Segments $A, C, E$ and $G$ are {\em edge segments\/} and  $B, D, F$ and $H$ are {\em rose segments}. In the counterclockwise orientation, the directions of edges $A$ and $C$ are $0$  and $k$, respectively.

Recall that, for odd $n$, the edge segments bisect the sequence
$$
\alp{n}=1\ \vline\ 31\ \vline\ 531\ \vline\ \dots \
\vline\ (n-4)(n-6)\dots 31\ \vline\, \vline\
13\dots (n-6)(n-4)\ \vline\  \dots \ \vline\
135\ \vline\ 13\ \vline\ 1
$$
of unit rhombuses, where shape $(m,n-m)$ is represented simply as $m$, and the vertical lines are added to emphasize the structure of the
sequence. When $n$ is even, the sequence of bisected unit rhombuses is
$$
\alp{n}=0\ \vline\ 20\ \vline\ 420\ \vline\ \dots \
\vline\ (n-4)(n-6)\dots 20 \ \vline\, \vline\ 02 \dots (n-6)(n-4)\ \vline\  \dots \ \vline\ 024\ \vline\ 02\ \vline\ 0
$$
where $0$ denotes a single unit edge along the edge segment. The sequences are the non-underlined parts of $\Sigma(n)$, as shown in
Tables~\ref{tab:composition} and \ref{tab:composition2}.
Remark that the sequences are palindromes, and that
all numbers $m$ in the sequences have the same parity as $n$.
Let
$$
f(m)=\frac{n-m}{2}-1,
$$
so that $m$  appears $f(m)$ times in both halves of $\alp{n}$.

Observe the following structure of $\alp{n}$:
let $k,m\in\{1,3,\dots ,n-2\}$ or $k,m\in\{0,2,\dots ,n-2\}$ when $n$ is odd or even, respectively, and $k<m$.
Then the projected sequence obtained by erasing all other numbers from $\alp{n}$ has the form
\begin{equation}
\label{eq:edgesegment}
k^i (m k)^j (k m)^j k^i, \mbox{ where $i=f(k)-f(m)\geq 1$ and $j=f(m)\geq 0$.}
\end{equation}

On the horizontal segment $A$, each occurrence of $m$ in $\alp{n}$  contributes symbols $\frac{m}{2},\frac{-m}{2}$ on the boundary word, when  $m\neq 0$. This is because $A$ is in direction $0$, so that
each $m$ represents a unit rhombus $(m,n-m)$ whose edges are
in directions $\frac{m}{2}$ and $\frac{-m}{2}$.
Each occurrence of $m=0$ contributes a single $0$ on the boundary word.
When $n$ is even, all symbols on the boundary word are integers; when $n$ is odd they are
from the set $\Z+\frac{1}{2}$. This is easily seen to be the case on all eight
segments in Figure~\ref{fig:boundarysegments}, including both the edge and the rose segments.
We have that in the case of odd $n$
the directions of the unit vectors present on the boundary are not parallel to any edge of $S$.
In contrast, on even $n$, each occurrence of $0$ in $\alp{n}$ yields a unit vector parallel to an edge.
In both even and odd cases, directions $\frac{\pm n}{2}$ perpendicular to $A$ are possible.

For any direction $x$, where $x\in\Z$ or $x\in \Z+\frac{1}{2}$ if $n$ is even or odd, respectively,
we introduce the notation $\diam{x}=m$ if the unit rhombus $(m,n-m)$ on segment $A$
contributes direction $x$ on the boundary word. More precisely,
$\diam{x}=m$ for all
directions $x$ such that $\pm 2x \equiv m \pmod{2n}$, and $0\leq m\leq n$.
We also define
$$
\s{x}=
\left\{
\begin{array}{lll}
x, & \mbox{ if $\pm 2x \equiv \diam{x}$} & \mbox{ $\pmod{4n}$, and}\\
\anti{x}, & \mbox{ if $\pm 2x \equiv \diam{x}+2n$} & \mbox{ $\pmod{4n}$},
\end{array}
\right.
$$
so that, among the two orientations $x$ and $\anti{x}$, the actual
contribution by rhombus $(\diam{x}, n-\diam{x})$ on the boundary word along segment $A$  is $\s{x}$.

Let us discuss next briefly the rose segments. The four
segments combined in the order $H, F, D, B$ form the full rose $R_2^1$ in the clockwise orientation.
Each unit vector direction appears twice in the rose. In $R_2$ the directions would appear perfectly
ordered in the increasing order, but omitting the outermost ring in $R_2^1$ swaps consecutive pairs
of directions. In segment $B$, directions $\frac{n}{2}$ and $k-\frac{n}{2}$ both appear once (as the second,
and the second last direction, respectively), while the directions $x$ in the interval
$k-\frac{n}{2} < x < \frac{n}{2}$ appear twice in $B$.

Let $a,b\in\Z$ (even $n$) or $a,b\in\Z+\frac{1}{2}$ (odd $n$)
be two distinct directions of unit vectors on the boundary word.
We denote by $\alpha(a,b)$, $\beta(a,b)$, $\gamma(a,b)$ and $\delta(a,b)$ the
$\pi_{a,b}$ projections of the boundary word segments on $A$, $B$, $C$ and $D$, respectively.
We denote their concatenation by $u$:
$$
u=\alpha(a,b)\beta(a,b)\gamma(a,b)\delta(a,b).
$$
The $\pi_{a,b}$ projection of the entire boundary is $u\anti{u}$.

It is easy to see that $\alpha(a,b)$ depends on the directions $a$ and $b$ as follows:
\begin{itemize}
\item If $2a\equiv n \pmod{2n}$ so that 
$a$ is orthogonal to segment $A$, then direction $a$ does not appear in $\alpha(a,b)$ and hence
$$
\alpha(a,b) = \s{b}^i,
$$
where $i=2f(\diam{b})=n-\diam{b}-2$.
Analogously, if $2b\equiv n \pmod{2n}$  then $\alpha(a,b)$ does not contain direction $b$.
\item Otherwise $\diam{a}, \diam{b} < n$, so we can infer $\alpha(a,b)$ from (\ref{eq:edgesegment}).
 If $\diam{a}<\diam{b}$ then
$$
\alpha(a,b)=\s{a}^{f(\diam{a})-f(\diam{b})}(\s{b}\s{a})^{f(\diam{b})} (\s{a}\s{b})^{f(\diam{b})}\s{a}^{f(\diam{a})-f(\diam{b})}.
$$
Notice that this word is a palindrome.
Analogously, if  $\diam{b}<\diam{a}$ then
$$
\alpha(a,b)=\s{b}^{f(\diam{b})-f(\diam{a})}(\s{a}\s{b})^{f(\diam{a})} (\s{b}\s{a})^{f(\diam{a})}\s{b}^{f(\diam{b})-f(\diam{a})}.
$$
The last possibility is $\diam{a}=\diam{b}$. Now
$$\alpha(a,b)=(\s{a}\s{b})^{2f(\diam{a})} \mbox{ or } \alpha(a,b)=(\s{b}\s{a})^{2f(\diam{a})}.$$
\end{itemize}

The word $\gamma(a,b)$ along segment $C$ is similar. Because the direction of $C$ is $k$, directions
$a$ and $b$ are oriented with respect to $C$ in the same way as $a-k$ and $b-k$ are oriented with respect to $A$.
We then have
$$
\gamma(a,b)=\sigma_k(\alpha(a-k,b-k)).
$$

\subsection{Case analysis}
\label{sec:caseanalysis}

We are interested to analyze the word $u=\alpha(a,b)\beta(a,b)\gamma(a,b)\delta(a,b)$ which is the first half of the projection of
the boundary word on directions $a$ and $b$, and to show that
$u\anti{u}$ satisfies the crossing condition.
There are a number of cases to analyze depending on the relationship between directions $a,b$ and $k$.
Note that we can reduce the number of cases due to symmetries. Because $a$ and $\anti{a}$, and $b$ and $\anti{b}$
define the same projections, we only need to consider one choice of each. We can also swap $a$ and $b$ if needed.
We can therefore assume that
$$-\frac{n}{2} < a< b \leq \frac{n}{2}.$$
With this choice, $ab\anti{a}\anti{b}$ is a proper rhombus oriented counterclockwise, and the rewrite rules in
(\ref{eq:rewrite}) can be used.

\bigskip

\noindent
{\bf Case 1: $a$ or $b$ is perpendicular to a side of super-rhombus $S$}
\medskip

By suitably orienting $S$, and possibly replacing $b$ by $\anti{b}$ or $a$ by $\anti{a}$, we can assume that $b=\frac{n}{2}$
and $-\frac{n}{2}<a<\frac{n}{2}$. We can further assume that $k\leq \frac{n}{2}$. (If not, we flip the rhombus
and consider $n-k$ instead of $k$.)
In this case we have $\alpha(a,b)=a^i$ for some i.

\medskip

\noindent
{\bf (a)} Assume $a>k-\frac{n}{2}$. Now $\beta(a,b)\in\{baa, aba\}$ and $\delta(a,b)=b$. In any case $\beta(a,b)\rewrite^*baa$.
Word $\gamma(a,b)$ is either some palindrome $p$ containing only letters $a$ and $b$ (if $\diam{a-k}\neq \diam{b-k}$), or $\gamma(a,b)=(ba)^j$ for some
$j$ (if $\diam{a-k}=\diam{b-k}$). In the first case
$$
u\derive^* a^i\ baa\ p\ b \derive^* b\ a^{i+2}\ p\ b\rewrite^* b\ p\ a^{i+2}\ b.
$$
The last two words are reversals of each other, so by Lemma~\ref{lem:lemma2} we know that the boundary word $u\anti{u}$ satisfies the crossing condition.
In the second case,
$$
u\derive^* a^i\ baa\ (ba)^j\ b \derive^* b\ a^{i+1}\ a(ba)^j\ b \rewrite^* b\ a(ba)^j\ a^{i+1}\ b.
$$
Again, the last two derived words are reversals of each other.

\medskip

\noindent
{\bf (b)} Next we assume that $a=k-\frac{n}{2}$. Now $\beta(a,b)=ba$, $\gamma(a,b)=b^j$ for some $j$, and $\delta(a,b)=\anti{a}b$, so that
$$
u=a^i\ ba\ b^j\ \anti{a}b \derive^* a^i\ b^{j+1}\ a\anti{a}\ b \derive a^i b^{j+2} \rewrite^* b^{j+2} a^i.
$$
The last two words are reversals of each other.

\medskip

\noindent
{\bf (c)} If  $-\frac{n}{2}<a<k-\frac{n}{2}$ then  $\beta(a,b)=b$ and $\delta(a,b)\in\{\anti{a}\anti{a}b,\anti{a}b\anti{a}\}$.
As $\diam{a-k}\neq\diam{b-k}$, we know that $\gamma(a,b)$ is some palindrome $p$ that contains letters $\anti{a}$ and $b$ only.
$$
u \derive^* a^i\ b\ p\ \anti{a}\anti{a}b \derive^* b\ a^i\ p\ \anti{a}\anti{a}b \derive^* a^i\ p\ \anti{a}\anti{a}
\derive^* p\ \anti{a}^{i+2}\rewrite^* \anti{a}^{i+2}\ p.
$$
Also here, the last two words are reversals of each other.

\medskip

\noindent
In the remaining cases we can now assume that $a$ and $b$ are not perpendicular to the sides of $S$.

\bigskip

\noindent
{\bf Case 2: $\diam{a}=\diam{b}$}
\medskip

Let us now assume that directions $a$ and $b$ are from the same unit rhombus on some edge segment.
In a suitable orientation of $S$ this means that
$a=-b$. By symmetries, we can also assume that $0 < b<\frac{n}{2}$ and $k\leq \frac{n}{2}$.
Now $\alpha(a,b)=(ba)^i$ for $i=2f(\diam{b})=n-2b-2$.

\medskip

\noindent
{\bf (a)} Assume $a>k-\frac{n}{2}$. We have $\beta(a,b)\in\{bbaa, baba\}$ and $\delta(a,b)=\varepsilon$, the empty word.
Because $\diam{a-k}=2|a-k|=2k+2b$ and $\diam{b-k}=2|b-k|$, it follows that $\diam{a-k}-\diam{b-k}$ is $4b$ or $4k$, depending on
whether $b<k$ or $b>k$, respectively. In any case, $\diam{a-k}>\diam{b-k}$, so we have from (\ref{eq:edgesegment}) that
$\gamma(a,b)=b^s(ab)^t(ba)^tb^s$  for some $s,t$. In fact, $s=\frac{1}{2}(\diam{a-k}-\diam{b-k})=\min\{2b,2k\}$.

If $s\geq 2$ then
$$
\begin{array}{rcl}
u &\derive^*& (ba)^i\ bbaa\ b^{s}(ab)^t(ba)^tb^s\\
& \derive^* &
b^s(ab)^{i+2}\ (ab)^t(ba)^tb^s \\
&
\derive^*
&
(ab)^{i+2}\\
&\rewrite^*&
(ba)^{i+2}.
\end{array}
$$
In the third step we eliminated $b^s(ab)^t$ and its reversal from the prefix and the suffix of the word, respectively.
The last words are reversals of each other so
Lemma~\ref{lem:lemma2} applies.

Suppose then $s<2$. Because $s=\min\{2b,2k\}$, we see that $s<2$ happens
if and only if $b=\frac{1}{2}$, $a=-\frac{1}{2}$. But in this case, the rose segment $B$ contains a unit vector
in direction $a$ between two unit vectors in direction $b$, so that $\beta(a,b)=baba$. (This is, interestingly, the reason why the
rose segments come from the rose $R_2^1$ instead of the full rose $R_2$: in the full rose $R_2$ we would have $bbaa$ instead of $baba$,
which would make the crossing condition fail in the case $b=\frac{1}{2}$, $a=-\frac{1}{2}$.) Now we have $s=1$ and
$$
u=(ba)^i\ baba\ b(ab)^t(ba)^tb = (ba)^{i+t+2}\ bb\ (ab)^t
\derive^* (ba)^{i+2}\ bb
 \rewrite^*  bb\ (ab)^{i+2},
$$
where the last two derived words are reverses of each other.

\medskip

\noindent
{\bf (b)} The other possibility is that $-\frac{n}{2}<a<k-\frac{n}{2}$. (Note that
the case $a=k-\frac{n}{2}$ is covered by Case 1 because then $a$ is perpendicular to segment $C$ of $S$.)
Now $\beta(a,b)=bb$ and $\delta(a,b)=\anti{a}\anti{a}$. Also we can calculate
$\diam{a-k}=2(a-k+n)=2(n-k-b)$ and $\diam{b-k}=2|b-k|$ so that
$\diam{a-k}-\diam{b-k} = 2\min\{n-2b,n-2k\}\geq 0$.

Assume first that  $k=\frac{n}{2}$ so that
$\diam{a-k}=\diam{b-k}$.
We have $\gamma(a,b)=(\anti{a}b)^j$ for some $j$, so
$$
\begin{array}{rcl}
u &=&(ba)^i\ bb\ (\anti{a}b)^j\ \anti{a}\anti{a}\\
& \derive^* &
(ba)^i\ (b\anti{a})^{j+2}
\end{array}
$$
The obtained word is of the form covered by
Lemma~\ref{lem:lemma3}.

Assume then that $k<\frac{n}{2}$ so that $\diam{a-k} >\diam{b-k}$. Then
$\gamma(a,b)=b^s(\anti{a}b)^t(b\anti{a})^t b^s$  for some $s,t$. We derive
$$
\begin{array}{rcl}
u & = &  (ba)^i\ bb\ b^s(\anti{a}b)^t\ (b\anti{a})^tb^s\ \anti{a}\anti{a} \\
&\derive^*&  b^s\ (ba)^i\ bb\  (\anti{a}b)^t\ (b\anti{a})^t\ \anti{a}\anti{a}\ b^s\\
&\derive^*&  (ba)^i\ (b\anti{a})^{2t+2}.
\end{array}
$$
Lemma~\ref{lem:lemma3} applies to the derived word.
\medskip

\noindent
From now on we can assume that neither Case 1 nor Case 2 applies.

\bigskip

\noindent
{\bf Case 3: Directions $a$ and $b$ appear in the same rose segment $B$}

\medskip

Now $k-\frac{n}{2}<a<b<\frac{n}{2}$ and $0<k<n$. (Note that we allow $k>\frac{n}{2}$ so that we
do not assume the angle between segments $A$ and $B$ to be obtuse.)
We can also assume that $b>0$ because otherwise we can reflect the super-rhombus to swap segments $A$ and $C$.

In this setup we have $\beta(a,b)\in\{bbaa, baba\}$ and $\delta(a,b)=\varepsilon$. Moreover,
$\alpha(a,b)$ and $\gamma(a,b)$ are palindromes containing only letters $a$ and $b$,
and $\diam{a}=2|a|$, $\diam{b}=2|b|=2b$, $\diam{a-k}=2|a-k|$ and $\diam{b-k}=2|b-k|$.
Because Case 2 does not apply, we have
$\diam{a}\neq \diam{b}$ and $\diam{a-k}\neq \diam{b-k}$.

\medskip

\noindent
{\bf (a)} Assume first that $\diam{a} >\diam{b}$. Then $|a|>b>0$ but $a<b$, which clearly implies that $a<0$.
Then also
$$
\diam{a-k} = 2|a-k| = 2(|a|+k) > 2(b+k) \geq 2|b-k| = \diam{b-k}.
$$
Now
$\alpha(a,b)=b^i(ab)^j(ba)^jb^i$,  and
 $\gamma(a,b)=b^s(ab)^t(ba)^tb^s$ for some $i,s\geq 1$ and $j,t\geq 0$.
 More precisely,
 $$
 \begin{array}{rcl}
 s&=&f(\diam{b-k})-f(\diam{a-k})=|a-k|-|b-k|= k-a-|b-k|,\\
 i &=&f(\diam{b})-f(\diam{a})=|a|-|b|=-a-b.
 \end{array}
 $$
In particular,
$s-i=k+b-|b-k|> 0$ so that $s> i\geq 1$.
Also, because $\diam{a-k}=-2(a-k)>-2a=\diam{a}$, we have that $t<j$.
We derive
$$
\begin{array}{rcl}
u &\derive^* & b^i(ab)^j(ba)^jb^i\ bbaa\ b^s(ab)^t(ba)^tb^s\\
&\derive^* &
(ab)^j(ba)^jb^{i+2}\ abab\ b^{s-2}(ab)^t(ba)^tb^{s-i}\\
&\derive^* &
(ab)^j(ba)^jb^{2s}\ abab\  (ab)^t(ba)^t\\
&\derive^* &
(ab)^{j-t}(ba)^jb^{2s}\ (ab)^{t+2}\\
&\derive^* &
(ba)^{2j-t}b^{2s}\ (ab)^{t+2}\\
&\derive^* &
(ba)^{2j-2t-2}b^{2s}\\
&\rewrite^* &
b^{2s}(ab)^{2j-2t-2}.
\end{array}
$$
The last two words are reversals of each other.

\medskip

{\bf (b)} Assume next that $\diam{a} <\diam{b}$ and $\diam{a-k}<\diam{b-k}$.
Then $k-b<k-a\leq |k-a| < |k-b|$  so that $k-b<0$, that is, $b>k$.

Now $\alpha(a,b)=a^i(ba)^j(ab)^ja^i$,
 for $i=f(\diam{a})-f(\diam{b})$ and $j=f(\diam{b})$,
 and $\gamma(a,b)=a^s(ba)^t(ab)^ta^s$, for $s=f(\diam{a-k})-f(\diam{b-k})$ and $t=f(\diam{b-k})$.
 We have
 $$
 \begin{array}{rcl}
 t-j &=& |b|-|b-k| = b-(b-k)=k>0, \mbox{ and }\\
 i-s &=& |b|-|a|-(|b-k|-|a-k|)\\ &=& b-(b-k)+|a-k|-|a| \geq k + (|a|-k)-|a| =0,
 \end{array}
 $$
so that $t>j$ and $i\geq s\geq 1$.

Assume first that $i\geq 2$. Then
$$
\begin{array}{rcl}
u &\derive^*& a^i(ba)^j(ab)^ja^i\ bbaa\ a^s(ba)^t(ab)^ta^s\\
&\derive^*&
a^{i-s}(ba)^j(ab)^ja^i\ bbaa\ a^s(ba)^t(ab)^t\\
&\derive^*&
(ba)^j(ab)^j\ a^{2i-s}\ bbaa\ a^s(ba)^t(ab)^t\\
&\derive^*&
(ab)^{j+2}\ a^{2i} (ba)^t(ab)^{t-j}\\
&\derive^*&
(ab)^{j+2}\ a^{2i}\ (ba)^{2t-j}\\
&\derive^*&
a^{2i}\ (ba)^{2t-2j-2}\\
&\rewrite^*&
(ab)^{2t-2j-2} a^{2i}.
\end{array}
$$
Note that we used the fact that $2i-s\geq 2$ on the fourth derivation line to change $aabb$ into $abab$.
The last two words are reversals of each other.

Consider then the case $i=1$, so that also $s=1$.
 But $i=b-|a|=1$ and $s=b-k-|a-k|=1$ happen simultaneously only if
 $b=a+1$. In this case, $\beta(a,b)=baba$
as the path follows rose $R_2^1$ instead of
 rose $R_2$. We derive
$$
\begin{array}{rcl}
u &=& a(ba)^j(ab)^ja\ baba\ a(ba)^t(ab)^ta\\
&\derive^*&
(ab)^{j+2}\ aa\ (ba)^t(ab)^{t-j}\\
&\derive^*&
(ab)^{j+2}\ aa\ (ba)^{2t-j}\\
&\derive^*&
aa\ (ba)^{2t-2j-2}\\
&\rewrite^*&
(ba)^{2t-2j-2}\ aa,\\
\end{array}
$$
obtaining a word and its reversal. Note that also here it was essential that rose $R_2^1$ was used instead of
rose $R_2$.

\medskip

\noindent
{\bf (c)} The last possibility is that $\diam{b} >\diam{a}$ and $\diam{b-k}<\diam{a-k}$. Then
$\alpha(a,b)=a^i(ba)^j(ab)^ja^i$  and
 $\gamma(a,b)=b^s(ab)^t(ba)^tb^s$ for some $i,s\geq 1$ and $j,t\geq 0$. We have
$$
\begin{array}{rcl}
u &\derive^* & a^i(ba)^j(ab)^ja^i\ bbaa\ b^s(ab)^t(ba)^tb^s\\
&\derive^* &
a^{i-1}(ab)^{2j}a^{i-1}\ aa\ bbaa\ bbb^{s-1}(ab)^{2t}b^{s-1}\\
&\derive^* &
a^{i-1}b^{s-1}(ab)^{2j}aba\ baba\ b(ab)^{2t}b^{s-1} a^{i-1}\ \\
&\derive^*&
(ab)^{2j+2t+4} \\
&\rewrite^* &
(ba)^{2j+2t+4}.
\end{array}
$$
The last two words are reversals of each other.
\medskip

\bigskip

\noindent
{\bf Case 4: None of the cases 1--3 apply}

\medskip

\noindent
Because Case 3 does not apply, directions $a$ and $b$ do not appear on the same rose segment; because Case 1 does not apply
the directions appear on unique rose segments, with the antiparallel directions then appearing on the opposite rose segments.
We can orient rhombus $S$ so that $b$ appears in segment $B$ and $\anti{a}$ in segment $D$. Then
$\beta(a,b)=bb$ and $\delta(a,b)=\anti{a}\anti{a}$.
Now $\alpha(a,b)$ and $\gamma(a,b)$ are palindromes containing only letters $a$ and $b$,
and only letters $\anti{a}$ and $b$, respectively. Then
$$w=bb\ \gamma(a,b)\ \anti{a}\anti{a} = \beta(a,b)\gamma(a,b)\delta(a,b)$$
contains only letters $\anti{a}$ and $b$.

Depending on whether $\diam{a}>\diam{b}$ or $\diam{a}<\diam{b}$ we have either
$\alpha(a,b)=b^i(ab)^j(ba)^jb^i$ or
$\alpha(a,b)=a^i(ba)^j(ab)^ja^i$, for some $i\geq 1$ and $j\geq 0$.
In the first case,
$$
\begin{array}{rcl}
u & = & b^i(ab)^j(ba)^jb^i\ w\\
&\derive^* & b^i (ba)^{2j}\ w\ b^i\\
&\derive^* & (ba)^{2j}\ w,
\end{array}
$$
and also in the second case
$$
\begin{array}{rcl}
u & = & a^i(ba)^j(ab)^ja^i\ w\\
&\derive^* & (ba)^{2j}a^i\ w\ \anti{a}^i\\
&\derive^* & (ba)^{2j}a^i\anti{a}^i\ w\\
&\derive^* & (ba)^{2j}\ w.\\
\end{array}
$$
We see that in any case, $u\derive^*(ba)^{2j}\ bb\ \gamma(a,b)\ \anti{a}\anti{a}$.

Analogously, depending on whether
$\diam{a-k}>\diam{b-k}$ or $\diam{a-k}<\diam{b-k}$ we have either
$\gamma(a,b)=b^s(\anti{a}b)^t(b\anti{a})^tb^s$ or
$\gamma(a,b)=\anti{a}^s(b\anti{a})^t(\anti{a}b)^t\anti{a}^s$, for some $s\geq 1$ and $t\geq 0$.
In the first case,
$$
\begin{array}{rcl}
u &\derive^* & (ba)^{2j}\ bb\ b^s(\anti{a}b)^t(b\anti{a})^tb^s\ \anti{a}\anti{a} \\
&\derive^* & b^s\ (ba)^{2j}\ (b\anti{a})^{2t+2}\ b^s\\
&\derive^* & (ba)^{2j}\ (b\anti{a})^{2t+2}.
\end{array}
$$
According to Lemma~\ref{lem:lemma3}, the crossing condition is satisfied.
In the second case,
$$
\begin{array}{rcl}
u &\derive^* & (ba)^{2j}\ bb\ \anti{a}^s(b\anti{a})^t(\anti{a}b)^t\anti{a}^s\ \anti{a}\anti{a} \\
&\derive^* & (ba)^{2j}\ \anti{a}^s\ (b\anti{a})^{2t+2}\ \anti{a}^s\\
&\derive^* & a^s\ (ba)^{2j}\ \anti{a}^s\ (b\anti{a})^{2t+2}\\
&\derive^* & (ba)^{2j}\ a^s\ \anti{a}^s\ (b\anti{a})^{2t+2}\\
&\derive^* & (ba)^{2j}\ (b\anti{a})^{2t+2},
\end{array}
$$
so again Lemma~\ref{lem:lemma3} applies.

\bigskip

This completes the case analysis. All possible cases were covered, and the crossing condition was confirmed in each case.
 This completes the proof of Theorem~\ref{maintheorem}.
\qed

\section{Conclusions}

We have demonstrated primitive substitutions on unit rhombuses that generate substitution tilings with $2n$-fold rotational symmetry for all $n$. The obtained tilings are uniformly recurrent, \emph{i.e}., quasiperiodic.  The proof was based on
a rewrite system on the proposed boundaries of the enlarged rhombuses to check that the interior can be properly tiled.

\begin{acknowledgements}
We would like to thank Henna Helander for editing the images and Reino Niskanen for many helpful comments and suggestions.
\end{acknowledgements}


%
%

\bibliographystyle{spmpsci}
\bibliography{subrosa}

\end{document}